\documentclass[10pt]{amsart}
 
\usepackage[margin=1in]{geometry}
\usepackage[english]{babel}

\usepackage{amssymb, amsmath, amsthm}
\usepackage{enumitem}
\usepackage{graphicx}

\usepackage[utf8]{inputenc}
 \usepackage{fourier}
\usepackage[T1]{fontenc}

\usepackage[naturalnames]{hyperref}

\usepackage[numbers]{natbib}

\newcommand{\ZZ}{\mathbb{Z}}
\newcommand{\RR}{\mathbb{R}}

\newcommand{\mmat}[1]{\left(\begin{smallmatrix} #1 \end{smallmatrix}\right)}
\newcommand{\choice}[1]{\left\{\begin{matrix} #1 \end{matrix}\right.}
\newcommand{\mset}[2]{\left\{ #1 \;\middle|\; #2  \right\}}
\newcommand{\sprod}[1]{\left\langle #1 \right\rangle}
\newcommand{\floor}[1]{\left\lfloor #1 \right\rfloor}

\newcommand{\cone}{\operatorname{\mathsf{cone}}}
\newcommand{\mmod}{\operatorname{mod}}
\newcommand{\mdiv}{\operatorname{div}}

\newcommand{\rar}{\rightarrow}

\newcommand{\PPP}{\mathcal{P}}

\renewcommand{\geq}{\geqslant}
\renewcommand{\leq}{\leqslant}

\newtheorem{theorem}{Theorem}

\newtheorem{lemma}[theorem]{Lemma}

\theoremstyle{remark}

\begin{document}

\title{A polyhedral model of partitions with bounded differences and \\ a bijective proof of a theorem of Andrews, Beck, and Robbins}

\author{Felix Breuer}
\address{Felix Breuer\\
Research Institute for Symbolic Computation (RISC)\\
Johannes Kepler University\\
Altenbergerstra{\ss}e 69\\
A-4040 Linz, Austria}
\email{felix@felixbreuer.net}
\urladdr{http://www.felixbreuer.net}

\author{Brandt Kronholm}
\address{Brandt Kronholm\\
Research Institute for Symbolic Computation (RISC)\\
Johannes Kepler University\\
Altenbergerstra{\ss}e 69\\
A-4040 Linz, Austria}
\email{kronholm@risc.jku.at}

\thanks{Both authors were supported by Austrian Science Fund (FWF) special research group \emph{Algorithmic and Enumerative Combinatorics} SFB F50, project number F5006-N15.}

\maketitle

\begin{abstract}
The main result of this paper is a bijective proof showing that the generating function for partitions with bounded differences between largest and smallest part is a rational function.  
This result is similar to the closely related case of partitions with fixed differences between largest and smallest parts which has recently been studied through analytic methods by Andrews, Beck, and Robbins.  
Our approach is geometric: 
We model partitions with bounded differences as lattice points in an infinite union of polyhedral cones.  Surprisingly, this infinite union tiles a single simplicial cone. 
This construction then leads to a bijection that can be interpreted on a purely combinatorial level.  
\end{abstract}


\section{Introduction}

A partition of a non-negative integer $n$ is a weakly non-increasing finite sequence of positive whole numbers $\lambda_1\geq\lambda_2\geq\cdots\geq\lambda_k>0$ such that 
\[ n=\lambda_1+\lambda_2+\cdots+\lambda_k. \]
The integers $\lambda_1,\lambda_2,\cdots,\lambda_k$ are called the parts of the partition.  
We write $|\lambda|=n$ to denote the particular non-negative integer $n$ that $\lambda$ partitions. For $0\leq t\in \ZZ$, we say a partition $\lambda$ has \emph{bounded difference $t$} if the difference between the largest and smallest part of $\lambda$ is at most $t$. We use $\PPP_t$ to denote the set of all (non-empty) partitions with bounded difference $t$ and let $\PPP_t(n):=\PPP_t\cap\mset{\lambda}{|\lambda|=n}$. Furthermore, we let $p(n,t):=\#\PPP_t(n)$ and $P_t(q) := \sum_{n\geq 1} p(n,t)q^n = \sum_{\lambda\in\PPP_t} q^{|\lambda|}$ denote the corresponding counting and generating functions, respectively. 

The natural expression for $P_t(q)$ is the infinite sum of rational functions
\begin{eqnarray}
\label{eqn:inf-sum}
P_t(q) = \sum_{m\geq 1} \frac{ q^m }{(1-q^m)\cdot(1-q^{m+1})\cdot\ldots\cdot(1-q^{m+t})}.
\end{eqnarray}
This can be seen by classifying the partitions $\lambda \in\PPP_t$ by the size $m$ of its smallest part.
For fixed $m$, the partition $\lambda$ has to contain the part $m$ at least once and can contain any of the parts $m+1,\ldots,m+t$ any non-negative number of times. In short
\begin{eqnarray}
\label{eqn:inf-union}
  \PPP_t = \bigcup_{m\geq 1} \mset{(m+t)^{k_t} + \ldots + m^{k_0}}{k_0 \geq 1\text{ and } k_1,\ldots,k_t \geq 0}
\end{eqnarray}
where we use the exponent notation to denote the multiplicity with which a part appears. This yields (\ref{eqn:inf-sum}) immediately.  
We will refer to (\ref{eqn:inf-union}) later in this paper.

Our point of departure for this paper is the surprising fact that the infinite sum of rational functions (\ref{eqn:inf-sum}) simplifies to the rational function (\ref{eqn:rat-fun}) given in Theorem~\ref{thm:bounded-differences} below. Theorem~\ref{thm:bounded-differences} is analogous to and motivated by a recent result of Andrews, Beck, and Robbins \cite{Andrews2014} in which an infinite sum of rational functions very similar to (\ref{eqn:inf-sum}) is reduced to a single rational function by way of $q-$series manipulations. The Andrews, Beck, and Robbins result is discussed in detail in Section~\ref{sec:bounded-vs-fixed}.

\begin{theorem}
\label{thm:bounded-differences}
For all $t\geq 1$,
\begin{eqnarray}
\label{eqn:rat-fun}
P_t(q) = \left(\frac{1}{(1-q)(1-q^2)\cdot\ldots\cdot(1-q^t)} - 1\right) \cdot \frac{1}{1-q^t}.
\end{eqnarray}
\end{theorem}

Our goal in this paper is to achieve a combinatorial and geometric understanding of this formula. In particular, a striking feature of all the rational functions appearing in (\ref{eqn:inf-sum}) and (\ref{eqn:rat-fun}) is that they have a form typically obtained from polyhedral cones \cite{BeckRobins}. This begs three questions:
\begin{enumerate}[label=\roman*)]
\item Is there a polyhedral model of $\PPP_t$ that makes the fact that $P_t(q)$ is rational readily apparent?
\item Is there a geometric reason why the infinite sum of rational functions (\ref{eqn:inf-sum}) simplifies to a single rational function (\ref{eqn:rat-fun})? \label{GEAquestion}
\item Is there a bijective proof of Theorem~\ref{thm:bounded-differences}?
\end{enumerate}

Our contribution in this paper is that we provide affirmative answers to each of these questions and explain the constructions involved from geometric and combinatorial points of view; thereby providing an answer for question (\ref{GEAquestion} put to the authors by George Andrews.  In particular, we develop a polyhedral model of partitions with bounded differences that allows us to interpret the identity of (\ref{eqn:inf-sum}) and (\ref{eqn:rat-fun}) in terms of a tiling of a polyhedral cone in Theorem~\ref{thm:main-theorem}. This geometric result immediately implies Theorem~\ref{thm:bounded-differences} using different methods than the $q$-series manipulations employed in \cite{Andrews2014}. More importantly, the geometric approach then leads us to a bijective proof of Theorem~\ref{thm:bijection}, which is a combinatorial restatement of Theorem~\ref{thm:bounded-differences}. Even though our motivation for Theorem~\ref{thm:bijection} is geometric, our bijective proof is entirely combinatorial. 

In this paper we draw freely on notions from both partition theory and polyhedral geometry. For references on these subjects we refer the reader to the textbooks \cite{andrews1998theory,berndt2006number,BeckRobins,DeLoera2012,Ziegler}.

\section{$P_0(q)$ is not a Rational Function}
\label{sec:t=0}

For $t=0$ we have $P_0(q)= \sum_{m\geq 1} \frac{q^m}{1-q^m}$. In contrast to the cases where $t$ is one or greater, $P_0(q)$ is \emph{not} a rational function. In fact, $p(n,0) = d(n)$ for $n\geq 1$ where $d(n)$ counts the divisors of $n$. As a warm-up for the constructions below we will now show this fact via a simple polyhedral model.

We take our cue from (\ref{eqn:inf-union}) and write $\PPP_0$ as an infinite union of (open) rays in the plane:

\[
  X_0 := \bigcup_{m \geq 1} \mset{ \mu\mmat{1\\m-1} }{0 < \mu\in\RR}.
  \]

Note that the rays $\mset{ \mu\mmat{1\\m-1} }{0 < \mu\in\RR}$ can be viewed as half-open cones $C_m$, a perspective we will make use of below. Given the above definition, an integer point $x\in \ZZ^2 \cap X_0$ is of the form $x=\mmat{k\\k(m-1)}$. A bijection $\phi: \ZZ^2 \cap X_0 \rar \PPP_0$ can be defined by mapping $x=\mmat{k\\k(m-1)}$ to the partition $\lambda=m^k$. Let $H_n = \mset{x\in\RR^2}{ \sum_i x_i = n}$ denote the hyperplane of all points with coordinate sum $n$, or, for short, at \emph{height} $n$. Then $p(n,0) = \#\ZZ^2\cap H_n \cap X_0$.

\begin{figure}[t]
\begin{center}
\includegraphics[angle=0,width=6cm]{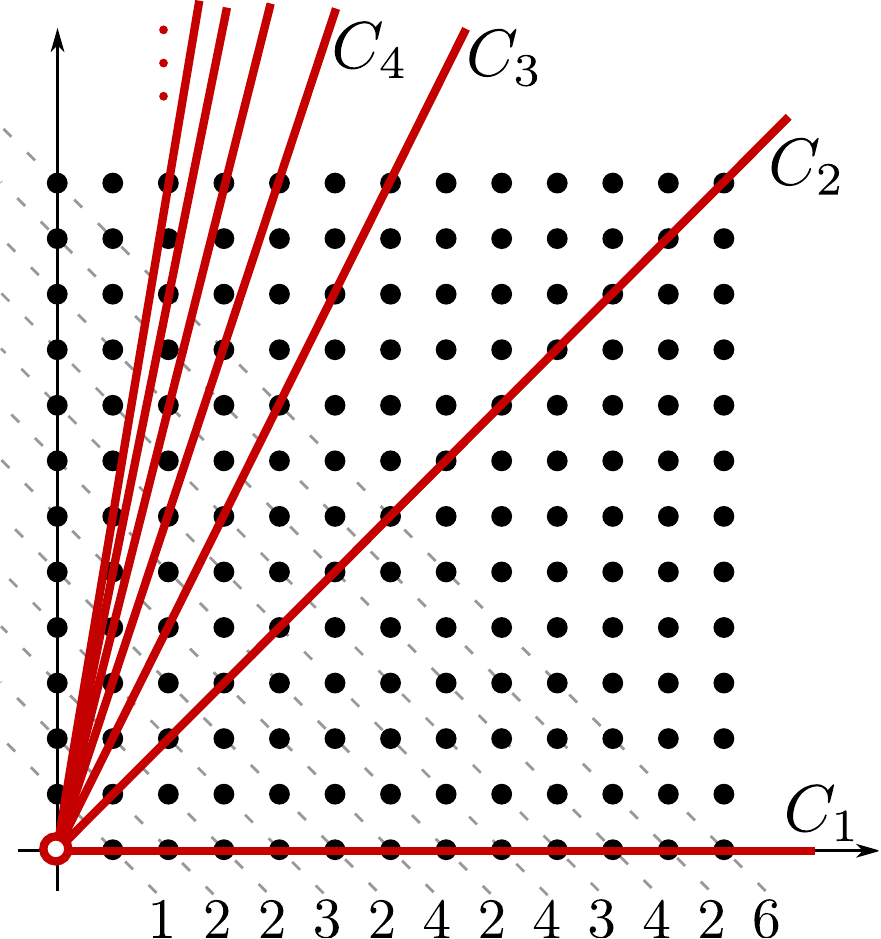}
\end{center}
 \caption[]{\label{fig:case-0} The polyhedral model for case $t=0$ is the union of relatively open rays (not including the origin) passing through the lattice points $\mmat{1\\i}$ for $i\in\ZZ_{\geq0}$. The number of lattice points contained in the union of these rays at different heights is given at the bottom of the figure.}
\end{figure}

This geometric model $X_0$ is illustrated in Figure~\ref{fig:case-0}. It corresponds to lifting $P_0(q)$ to the multivariate generating function $P_0(x_1,x_2)$ defined by 
\begin{eqnarray*}
  P_0(x_1,x_2)&=& \sum_{m\geq 1} \frac{x_1^{1}x_2^{m-1}}{1-x_1^1x_2^{m-1}}.
\end{eqnarray*}
In the following we will use multi index notation and write $x^v:=x_1^{v_1}\cdot\ldots\cdot x_d^{v_d}$ for a vector $v\in\ZZ^d$ so that the above equation reads
\begin{eqnarray*}
  P_0(x)&=& \sum_{m\geq 1} \frac{x^{\mmat{1\\m-1}}}{1-x^{\mmat{1\\m-1}}}.
\end{eqnarray*}
Note that $P_0(x)$ does indeed specialize to $P_0(q)$ by substituting $x_1=x_2=q$. To see that $p(n,0) = d(n)$ we observe that the ray $\mset{ \mu\mmat{1\\m-1} }{0 < \mu\in\RR}$ contains a lattice point at height $n$ if and only if there exists an integer $k$ such that $n=mk$, i.e., if and only if $m|n$. Since we sum over all $m\geq 1$ it follows that $p(n,0) = d(n)$.

\section{$P_1(q)$ is a Rational Function}
\label{sec:t=1}

In contrast to $P_0(q)$, the generating functions $P_t(q)$ are rational functions for all $t\geq1$. To build intuition before tackling the general case, we first show this for $t=1$ via a geometric argument  illustrated in Figure~\ref{fig:case-1}. Our strategy is this: Just as in the case $t=0$, we identify each set in the infinite union (\ref{eqn:inf-union}) as the set of lattice points in a half-open polyhedral cone $C_m$. However, in contrast to the case $t=0$, these cones $C_m$ are now 2-dimensional and they tile the half-open quadrant $\RR_{>0}\times\RR_{\geq 0}$, which is itself a single half-open simplicial cone. This immediately allows us to read off the desired rational function expression (\ref{eqn:rat-fun}) for $P_1(q)$. 

We now introduce some notation to make this argument precise. Given $o\in\{0,1\}^d$ and a matrix $V\in\ZZ^{d\times d}$ with linearly independent columns $v_1,\ldots,v_d$, we use
\[
  \cone^o(V) := \mset{\sum_{i=1}^d \mu_i v_i}{0 \leq \mu_i\in\RR \text{ and if $o_i=1$ then }0 < \mu_i}
\]
to denote the half-open simplicial cone generated by the columns of $V$ and where the facet opposite $v_i$ is open if and only if $o_i = 1$. With this notation we define
\begin{eqnarray}
\label{eqn:X_1}
   X_1 := \bigcup_{m \geq 1} C_m, \;\;\;\; \text{ and } \;\;\;\; C_m:=\cone^{(1,0)}(\mmat{1 & 1 \\ m-1 & m}) 
\end{eqnarray}
The generating function of the lattice points in $C_m$ satisfies
\[
\sum_{z\in \ZZ^2 \cap C_m} x^z = \frac{x^{\mmat{1\\m-1}}}{(1-x^{\mmat{1\\m-1}})(1-x^{\mmat{1\\m}})}.
\]
Substituting $x_1=x_2=q$ we obtain precisely the $m$-th summand in $P_1(q)$. This shows that (\ref{eqn:X_1}) is a polyhedral model of the expressions (\ref{eqn:inf-sum}) and (\ref{eqn:inf-union}). It follows that $p(n,1) = \#\ZZ^2\cap H_n \cap X_1$.

\begin{figure}[t]
\begin{center}
\includegraphics[angle=0,width=6cm]{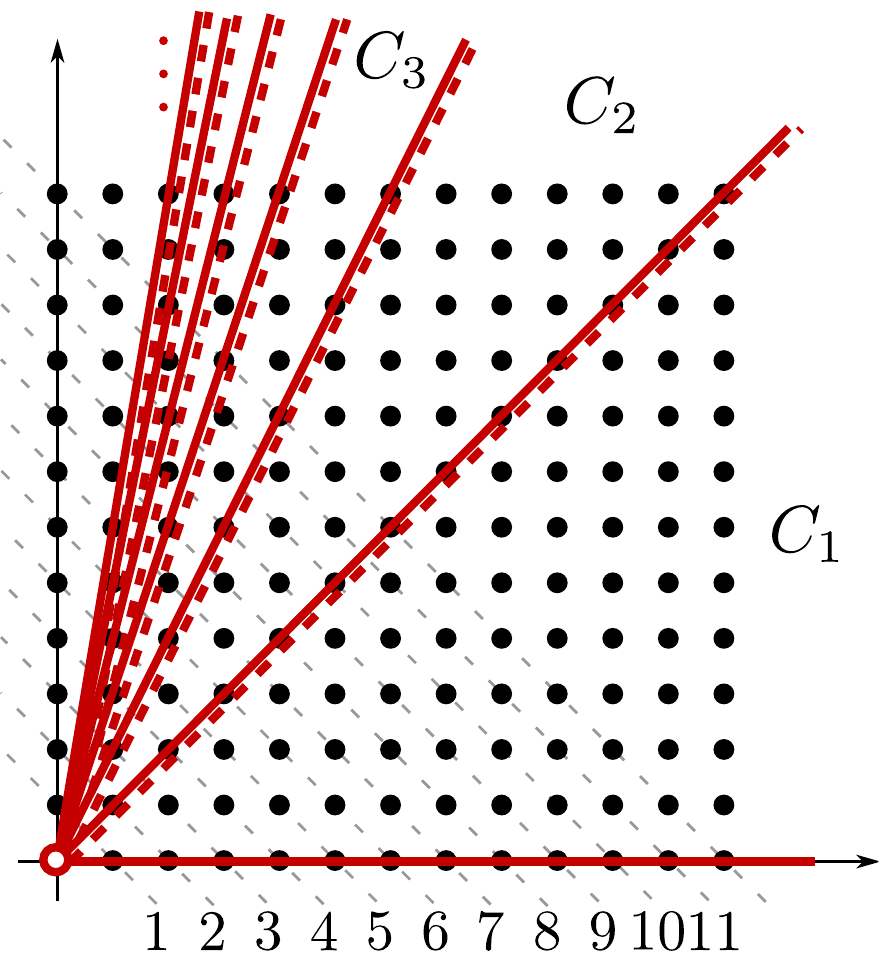}
\end{center}
 \caption[]{\label{fig:case-1} The polyhedral model for the case $t=1$ is the entire non-negative quadrant without the vertical axis. It is given as a union over all $i\in\ZZ_{\geq0}$ of the half-open cones with generators $\mmat{1\\i}$ and $\mmat{1\\i+1}$ whose top edge is open. This construction immediately shows that the lattice point count at different heights is a linear function.}
\end{figure}

As we can see in Figure~\ref{fig:case-1}, the cones $C_m$ tile the positive quadrant, excluding the vertical axis, i.e.,
\begin{eqnarray}
\label{eqn:X_1single}
  X_1 = \RR_{>0}\times\RR_{\geq 0} = \cone^{(1,0)}(\mmat{1 & 0 \\ 0 & 1}).
\end{eqnarray}
We have thus observed that the infinite union of cones (\ref{eqn:X_1}) is in fact a single simplicial cone (\ref{eqn:X_1single}). This allows us to read off the generating function immediately, namely
\[
  \sum_{v\in \ZZ^2 \cap \cone^{(1,0)}(\mmat{1 & 0 \\ 0 & 1})} x^v  = \frac{x^{\mmat{1\\0}}}{(1-x^{\mmat{1\\0}})(1-x^{\mmat{0\\1}})}
\]
from which, by substituting $x_1=x_2=q$, we obtain
\[
  P_1(q) = \frac{q}{(1-q)^2}
\]
which implies in particular
\[
  p(n,1) = \#\ZZ^2\cap H_n \cap \cone^{(1,0)}(\mmat{1 & 0 \\ 0 & 1}) = n.
\]
In this way, (\ref{eqn:X_1single}) can be viewed as a polyhedral model of (\ref{eqn:rat-fun}). We have thus shown Theorem~\ref{thm:bounded-differences} in the case $t=1$ via the geometric tiling argument shown in Figure~\ref{fig:case-1}. It turns out that this works for all $t\geq 1$, as we will see in Section~\ref{sec:polyhedral-model}. Interestingly, the bijection between $\ZZ^2\cap X_1$ and $\PPP_1$ implicit in this construction is non-trivial, as we discuss in Section~\ref{sec:bijection}.

\section{Polyhedral model in the general case $t\geq 1$}
\label{sec:polyhedral-model}

To handle the general case $t\geq 1$ we will need to add an additional twist to the construction, which we will illustrate by using $t=2$ as a running example.  
As before, the $m$-th summand in (\ref{eqn:inf-sum}) will correspond to a half-open $(t+1)$-dimensional simplicial cone $C_m$ in $\RR^{t+1}$.  
The cones $C_m$ are pairwise disjoint.   
Their union tiles a single simplicial cone $X_t\subset\RR^{t+1}$, which has one extreme ray removed.

\begin{figure}[t]
\begin{center}
\includegraphics[angle=0,width=6cm]{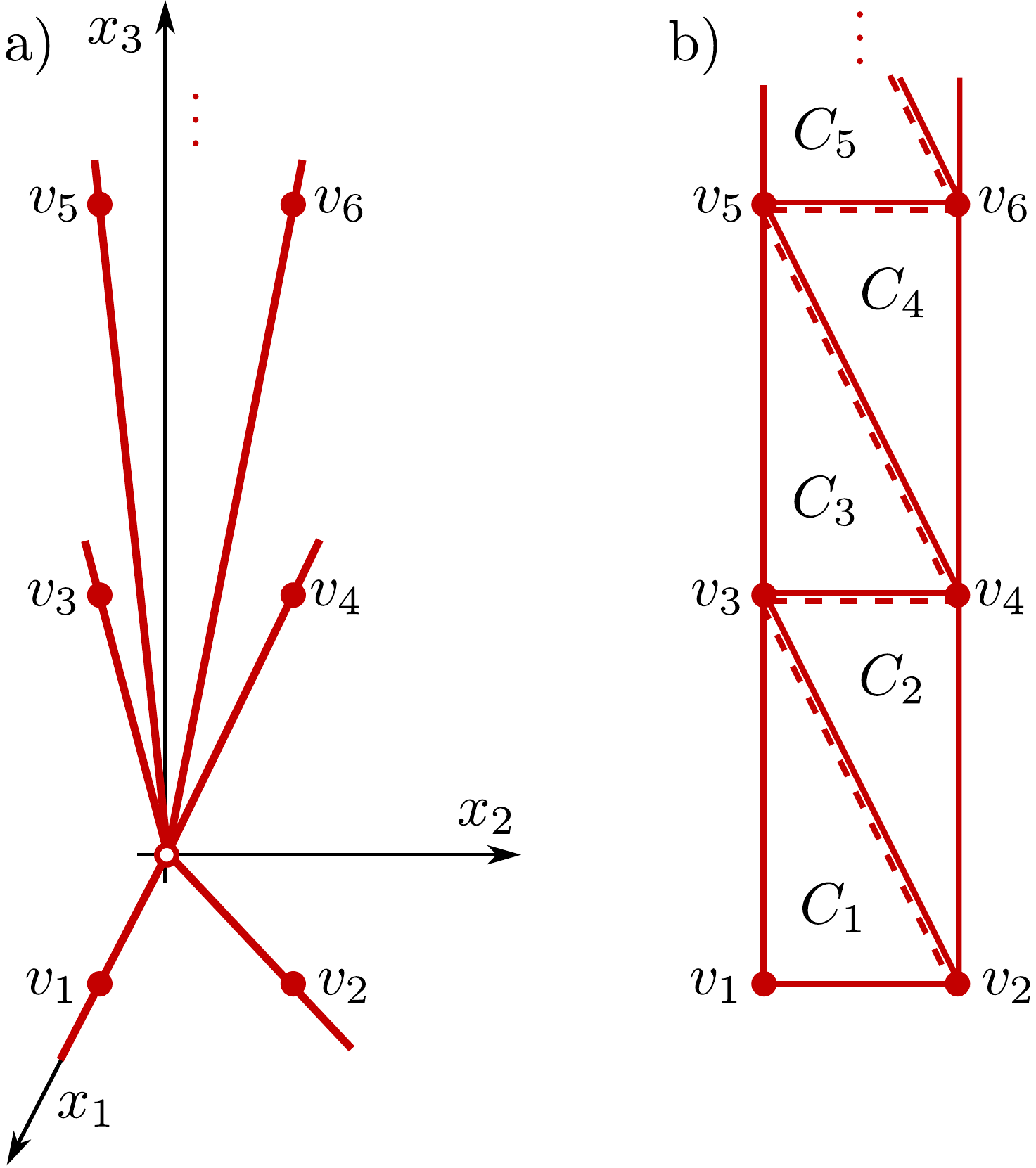}
\end{center}
 \caption[]{\label{fig:case-2} The polyhedral model for the case $t=2$. The generators $v_i$ and the extreme rays through the $v_i$ are shown in a). The cones $C_i$ are generated by vectors $v_i,v_{i+1},v_{i+2}$. The triangles in b) indicate which faces of the $C_i$ are open: each triangle is the intersection of the corresponding cone with the hyperplane given by $x_1=1$. 
The union of all $C_i$ is the closed cone generated by $v_1=e_1$, $v_2$ and $e_3$, with the vertical axis removed.  The $e_i$ are the standard unit vectors. }
\end{figure}

We begin this construction by defining vectors $b_0,\ldots,b_{t-1}\in\RR^{t}$ as follows: $b_j$ is the vector that contains $j+1$ leading ones and after that only zeros. We then define an infinite sequence of vectors $v_1,v_2,\ldots \in \ZZ^{t+1}$ by 
\[
  v_i := \mmat{b_{i-1 \mmod t} \\ ( i-1 \mdiv t )t}.
\] 
In the case $t=2$, this gives
\begin{eqnarray}\label{bvectors}
  v_1 = \mmat{1\\0\\0}, v_2 = \mmat{1\\1\\0}, v_3= \mmat{1\\0\\2}, v_4=\mmat{1\\1\\2}, v_5=\mmat{1\\0\\4}, \ldots, v_{2k+1}= \mmat{1\\0\\2k}, v_{2k+2} = \mmat{1\\1\\2k}, \ldots,
\end{eqnarray}
as shown in Figure~\ref{fig:case-2}. For each $i$ the sum of coordinates of $v_i$ is $|v_i|=i$. Let $V_m$ denote the matrix consisting of columns $v_m,v_{m+1},\ldots,v_{m+t}$ and define cones $C_m := \cone^{(1,0,\ldots,0)} V_m$, that is, the columns of $V_m$ are the generators of the cone $C_m$ and the facet opposite to the first generator is open. For all $m\geq1$, the columns of $V_m$ generate the same lattice $\Lambda$, which consists of all integer points where the last coordinate is divisible by $t$, i.e., $\Lambda := \ZZ^t \times t\ZZ = V_m\ZZ^{t+1}$ for all $m$. Let 
\begin{align}
  X_t &:= \bigcup_{m\geq 1} C_m = \bigcup_{m\geq 1} \cone^{(1,0,\ldots,0)} V_m = \bigcup_{m\geq 1} \mset{ \sum_{i=0}^t \alpha_i v_{m+i}}{ \alpha_i \geq 0, \alpha_0 > 0}
\end{align}
denote the (disjoint) union of these cones.
By construction
\[
  \sum_{z \in \Lambda \cap X_t} x^z = \sum_{m\geq 1} \left(\sum_{z \in \Lambda \cap C_m} x^z \right) =\sum_{m\geq 1} \frac{x^{v_m}}{(1-x^{v_m})\cdot\ldots\cdot(1-x^{v_{m+t}})}.
\]
Specializing $x_i=q$ we obtain precisely the $m$-th summand of $P_t(q)$. Therefore 
\begin{align}
\label{eqn:link-bounded-differences-and-geometry}
p(n,t)&=\#\Lambda \cap X_t \cap \mset{x\in\RR^{t+1}}{\sum x_i = n},
\end{align}
which means in particular that if we specialize $\sum_{z \in \Lambda \cap X_t} x^z$ at $x_0=x_1=\ldots=x_t=q$ we obtain $P_t(q)$.

For the case $t=2$, this construction is illustrated in Figure~\ref{fig:case-2}.  
As we can see, the $C_m$
tile the simplicial cone with generators $\mmat{1\\0\\0}$, $\mmat{1\\1\\0}$ and $\mmat{0\\0\\1}$, excluding the ray $\cone^0(\mmat{0\\0\\1})$, i.e.,  
\begin{eqnarray}
\label{eqn:X_2}
  X_2 = \cone^{0}(\mmat{1&1&0 \\ 0&1&0 \\ 0&0&1 }) \setminus \cone^0(\mmat{0\\0\\1}).
\end{eqnarray}

To obtain such a description of $X_t$ for all $t\geq1$, we provide an inequality description of the cones $C_m$ in the following Lemma.  

\begin{lemma} 
\label{lem:cone-description}
For all $m\geq 1$, $j\in\{0,...,t-1\}$, and $k\in\ZZ$
\begin{eqnarray}
\label{eqn:lemma X_t}
  C_m = \mset{x\in\RR^{t+1}}{x_0\geq \ldots \geq x_{t-1}\geq 0, \sprod{u_{m-1},x} \geq 0, \sprod{u_{m},x} < 0 }
\end{eqnarray}
where $e_0,...e_t$ are the standard basis vectors for $\RR^{t+1}$ and we define $u_{j,k}:=-kte_0+te_j+e_t$ and  $u_m :=u_{(m~\mmod~t),(i~\mdiv~t)+1}$.  For all $m$, the inequality $x_m \geq x_{m+1}$ found in (\ref{eqn:lemma X_t}) is redundant and can be omitted without changing the set $C_m$.  
The $(t+1)$-dimensional cone $C_m$ is therefore given by $(t+1)$ inequalities. Alternatively, the cones can also be defined by the infinite system of inequalities
\begin{eqnarray}
\label{eqn:lemma X_t infinite}
  C_i = \mset{x\in\RR^{t+1}}{x_0\geq \ldots \geq x_{t-1}\geq 0, \bigwedge_{l=0}^{i-1} \sprod{u_{l},x} \geq 0, \bigwedge_{l\geq i} \sprod{u_{l},x} < 0 }.
\end{eqnarray}
\end{lemma}

\begin{proof}
For $j\in\{0,\ldots,t-1\}$ and $k\in\ZZ$ we define the matrix 
\[
  M_{j,k} = \mmat{ 
    b_j     & \cdots & b_{t-1} & b_0 & \cdots & b_{j} \\
    kt  & \cdots & kt  & (k+1)t  & \cdots & (k+1)t 
  }.
\]
Every matrix $V_m$ is of the form $M_{j,k}$ for suitable $j$ and $k$. In fact, recalling (\ref{bvectors}) we have
\[
V_m = M_{(m-1~\mmod~t), (m-1~\mdiv~t)}.
\]
Each $V_m$ contains exactly $t+1$ linearly independent columns. 
To see that a given system $S$ of homogeneous linear inequalities is an inequality description of $V_m$, we must satisfy the following two conditions:
\begin{enumerate}[label=\roman*)]
\item The columns of $V_m$ must satisfy the inequalities in $S$\label{condition1}.
\item For every subset $F$ of $t$ columns from $V_m$ there is an inequality that is satisfied at equality by all $v\in F$. Such an inequality is called \emph{facet-defining}. \label{condition2}
\end{enumerate}
To see that the inequality system given in the statement of the lemma satisfies these conditions, we proceed as follows.

First, all columns appearing in any of the matrices satisfy the inequality system $x_0\geq \ldots \geq x_{t-1}\geq 0$.
Moreover, any vector of the form $\mmat{b_l \\ c}$, for any value $c$, satisfies $x_i = x_{i+1}$ for every $i\not= l$.  
This means that the inequality $x_i \geq x_{i+1}$ is facet-defining for $M_{j,k}$ for all $i\not=j$.  
In other words, for every $t$-subset $F$ of the columns of $M_{j,k}$ that contains both $\mmat{b_j\\kt}$ and $\mmat{b_j \\(k+1)t}$, we have found a facet-defining inequality. 

Next we show that the remaining two inequalities are facet-defining for the two $t$-subsets $F$ where exactly one of these two columns is omitted. To this end we compute, for $k,r\in\ZZ$, $j,l\in\{0,\ldots,t-1\}$,
\begin{eqnarray*}
\label{eqn:ujk}
  \sprod{u_{j,k},\mmat{b_l \\ rt}} &=& 
    \choice{  
      (r-k)t  & \text{ if $l<j$, }      \\ 
      (r-k+1)t & \text{ if $l\geq j$, }    
    }
\end{eqnarray*}
so that in particular
\begin{eqnarray*}
  \sprod{u_{j,k+1},\mmat{b_l \\ kt}} &=& 
    \choice{  
      -t  & \text{ if $l<j$, }      \\
      0   & \text{ if $l\geq j$, }    
    } \\
  \sprod{u_{j,k+1},\mmat{b_l \\ (k+1)t}} &=& 
    \choice{  
      0  & \text{ if $l<j$, }      \\ 
      t   & \text{ if $l\geq j$, }    
    }
\end{eqnarray*}
which implies, for $j\in\{0,\ldots,t-1\}$,
\begin{eqnarray}
\label{eqn:but-last}
  u_{j,(k+1)}^\top M_{j,k} &=& \mmat{ 
    0 & \cdots & 0 & t \\
  }
\end{eqnarray}
as well as, for $j\in\{0,\ldots,t-2\}$,
\begin{eqnarray}
\label{eqn:but-first-1}
  u_{j+1,(k+1)}^\top M_{j,k} &=& \mmat{ 
    -t & 0 & \cdots & 0  \\
  } \\
\label{eqn:but-first-2}
  u_{0,(k+2)}^\top M_{j,k} &=& \mmat{ 
    -t & 0 & \cdots & 0  \\
  }.
\end{eqnarray}
(\ref{eqn:but-last}) shows that, for all $m\geq 1$, the inequality $\sprod{u_{m-1 \mmod t, (m-1 \mdiv t) + 1},x} \geq 0$ defines the facet of the cone generated by $V_m = M_{m-1 \mmod t,m-1 \mdiv t}$ corresponding to the set $F$ consisting of all columns but the last.  
Similarly, (\ref{eqn:but-first-1}) and (\ref{eqn:but-first-2}) show that $\sprod{u_{m \mmod t, (m \mdiv t) + 1},x} < 0$ defines the (open) facet of the cone generated by $V_m = M_{m-1 \mmod t,m-1 \mdiv t}$ corresponding to the set $F$ consisting of all columns but the first.  
Thus conditions \ref{condition1} and \ref{condition2} are satisfied for the system of inequalities (\ref{eqn:lemma X_t}).

Finally, we see from (\ref{eqn:ujk}) that $\sprod{u_m,\mmat{b_l\\ rt}} \leq \sprod{u_{m'},\mmat{b_l\\ rt}}$ when $m \geq m'$. Therefore, $C_m$ satisfies all of the inequalities $\sprod{u_l,x}\geq 0$ for $l\leq m-1$ and all of the inequalities $\sprod{u_l,x} < 0$ for $l\geq m$. This shows that the additional inequalities given in the infinite system (\ref{eqn:lemma X_t infinite}) are redundant and therefore (\ref{eqn:lemma X_t infinite}) is correct as well.
\end{proof}

From this inequality description of the cones $C_m$, we can now derive a simple description of their union $X_t$.

\begin{theorem}
\label{thm:main-theorem}
For all $t\geq1$,
\begin{eqnarray*}
X_t &=& \mset{x\in\RR^{t+1}}{x_0 \geq \ldots \geq x_{t-1} \geq 0, x_t \geq 0} \setminus \mset{x\in\RR^{t+1}}{x_0 = \ldots = x_{t-1} = 0, x_t\geq 0} \\
    &=& \cone^0(\mmat{ 
            1 & 1 & \cdots    & 1      & 0 \\
              & 1 &           & 1      & 0 \\
              &   & \ddots & \vdots & \vdots \\
              &   &           & 1      & 0 \\
              &   &           &        & 1
        }) 
        \setminus \cone^0(\mmat{0 \\ \vdots \\ 0 \\\ 1}).
\end{eqnarray*}
\end{theorem}

\begin{proof}
Clearly, the difference of cones given in the theorem has the inequality description stated in the theorem. 
It remains to show the first equality asserted in the theorem.  
This follows from Lemma~\ref{lem:cone-description} by observing that consecutive cones $C_m$ and $C_{m+1}$ are ``glued together'' along the shared facet defined by $\sprod{u_{m},x} = 0$, which is open in the former cone and closed in the latter, with opposite orientations $\sprod{u_{m},x} <0$ and $\sprod{u_{m},x}\geq 0$.

Formally, we proceed by induction on $k$.  
For every $k\geq 1$ we have
\[
  \bigcup_{m=1}^k C_m = \mset{x\in\RR^{t+1}}{x_0 \geq \ldots \geq x_{t-1} \geq 0, \sprod{u_0,x} \geq 0, \sprod{u_k,x} < 0}.
\]
For $k=1$, this is the statement of Lemma~\ref{lem:cone-description}.  
Suppose the induction hypothesis holds true for some $k\geq 1$, it follows for $k+1$ by computing
\begin{eqnarray*}
 \bigcup_{m=1}^{k+1} C_m
 &=& \bigcup_{m=1}^{k} C_m \cup C_{k+1} \\ 
 &=& \mset{x\in\RR^{t+1}}{x_0 \geq \ldots \geq x_{t-1} \geq 0, \sprod{u_0,x} \geq 0, \sprod{u_k,x} < 0} \\
  && \cup \mset{x\in\RR^{t+1}}{x_0\geq \ldots \geq x_{t-1}\geq 0, \sprod{u_{k},x} \geq 0, \sprod{u_{k+1},x} < 0 } \\
  &=& \mset{x\in\RR^{t+1}}{x_0 \geq \ldots \geq x_{t-1} \geq 0, \sprod{u_0,x} \geq 0, \sprod{u_k,x} < 0, \sprod{u_{k+1},x} < 0} \\
  && \cup \mset{x\in\RR^{t+1}}{x_0\geq \ldots \geq x_{t-1}\geq 0, \sprod{u_0,x} \geq 0, \sprod{u_{k},x} \geq 0, \sprod{u_{k+1},x} < 0 } \\
  &=& \mset{x\in\RR^{t+1}}{x_0 \geq \ldots \geq x_{t-1} \geq 0, \sprod{u_0,x} \geq 0, \sprod{u_{k+1},x} < 0}
\end{eqnarray*}
where we use both the induction hypothesis and Lemma~\ref{lem:cone-description}.

Next, we observe that the conditions $\sprod{u_0,x} \geq 0$ and $\sprod{u_k,x} < 0$ reduce to
\[
  x_{t} \geq 0 \text{ and }  tx_{k \mmod t} + x_t < ((k \mdiv t) + 1)tx_0.
\]
Allowing $k\rar\infty$, the second constraint is reduced to the condition that $x_0>0$. Notice that given $x_0\geq \ldots \geq x_{t-1}\geq 0$ the condition $x_0 \leq 0$ implies $x_0=\ldots=x_{t-1}=0$.
\end{proof}

\section{Enumerative and Combinatorial Consequences}
\label{sec:bijection}

In the case $t=2$, our description (\ref{eqn:X_2}) of $X_2$ allows us to write the generating function of all $\Lambda$-lattice points in $X_2$ simply as
\[
  \sum_{z \in \Lambda \cap X_2} x^z = \frac{1}{(1-x_1x_2)(1-x_2)(1-x_3^2)} - \frac{1}{(1-x_3^2)} = \left(\frac{1}{(1-x_1x_2)(1-x_2)} - 1\right) \frac{1}{(1-x_3^2)}
\]
where we use the factor $(1-x_3^2)$ instead of $(1-x_3)$ since $\Lambda$ contains only those integer points with even last coordinate. Specializing $x_i=q$ we obtain
\[
  P_2(q) = \frac{1}{(1-q^2)^2(1-q)} - \frac{1}{(1-q^2)} = \frac{(1+q) - (1-q^2)^2}{(1-q^2)^3} = \frac{q+2q^2-q^4}{(1-q^2)^3}
\]
which yields
\begin{eqnarray}
  p(2k,2) &=& 0 \binom{k+2}{2} + 2\binom{k+1}{2} - 1 \binom{k}{2} \label{eqn:quasipolyP_2a}\\
  p(2k+1,2) &=& 1 \binom{k+2}{2} + 0\binom{k+1}{2} + 0 \binom{k}{2}.\label{eqn:quasipolyP_2b}
\end{eqnarray}

In just the same way, we can obtain Theorem~\ref{thm:bounded-differences} as an immediate corollary of Theorem~\ref{thm:main-theorem}.

\begin{proof}[Proof of Theorem~\ref{thm:bounded-differences}]
The generating function of
\begin{eqnarray*}
S_1 :=\Lambda\cap \cone^0(\mmat{ 
            1 & 1 & \cdots    & 1      & 0 \\
              & 1 &           & 1      & 0 \\
              &   & \ddots & \vdots & \vdots \\
              &   &           & 1      & 0 \\
              &   &           &        & 1
        })
&
\;\;\text{ is }\;\;
&
  \sum_{v\in S_1} x^v = \frac{1}{(1-x^{b_0})\cdot\ldots\cdot(1-x^{b_{t-1}})(1-x_t^t)}
\end{eqnarray*}
and the generating function of 
\begin{eqnarray*}
S_2 := \Lambda\cap \cone^0(\mmat{
  0 \\ \vdots \\ 0 \\\ 1
})
&
\;\;\text{ is }\;\;
  \sum_{v\in S_1} x^v = \frac{1}{1-x_t^t}.
\end{eqnarray*}
Applying Theorem~\ref{thm:main-theorem} we find
\[
  \sum_{z \in \Lambda \cap X_t} x^z = \frac{1}{(1-x^{b_0})\cdot\ldots\cdot(1-x^{b_{t-1}})(1-x_t^t)} - \frac{1}{1-x_t^t}.
\]
Due to (\ref{eqn:link-bounded-differences-and-geometry}) we can specialize $x_0=x_1=\ldots=x_t=q$ to obtain the desired identity
$$P_t(q) = \frac{1}{(1-q)(1-q^2)\cdot\ldots\cdot(1-q^t)^2} - \frac{1}{1-q^t}.$$ 
\end{proof}

Theorem~\ref{thm:main-theorem} not only implies this arithmetic corollary, but it moreover leads to a bijective proof of Theorem~\ref{thm:bounded-differences}: 
We can interpret (\ref{eqn:inf-sum}) as counting partitions with bounded differences and (\ref{eqn:rat-fun}) as counting pairs $(\lambda,\ell)$ where $\lambda$ is a non-empty partition with largest part at most $t$ and $\ell$ is a non-negative multiple of $t$.  
Our geometric construction of $X_t$ then leads directly to combinatorial bijection between these two classes.

\begin{theorem}
\label{thm:bijection}
For fixed $t\geq 1$ and any $n\in\ZZ_{\geq 1}$, the number $p(n,t)$ of partitions of $n$ with difference between largest and smallest part at most $t$ equals the number of pairs $(\lambda,\ell)$ where $\ell\in\ZZ_{\geq 0}$ is divisible by $t$ and $\lambda$ is a non-empty partition of $n-\ell$ with largest part at most $t$. Moreover, there is an explicit bijection between these sets.
\end{theorem}

Just as both the sum in (\ref{eqn:inf-sum}) and our geometric tiling in Theorem~\ref{thm:main-theorem} of $X_t$ with cones $C_m$ suggest, our proof of this result will proceed by constructing a bijection piece-by-piece. On one hand, for each $m\geq1$ the corresponding summand in (\ref{eqn:inf-sum}) is readily interpreted as the generating function of the set $\tilde{C}_m$ of partitions $\lambda$ with smallest part $m$ and difference between smallest and largest part at most $t$. 
On the other hand, we can infer a combinatorial interpretation of $C_m\cap\Lambda$ from the polyhedral model $C_m=\cone^{e_1}V_m$. Let $x=(x_0,\ldots,x_{t-1},x_{t})\in\Lambda\cap C_m$. From the inequality description of $C_m$ we know that $(x_0,\ldots,x_{t-1})$ is a weakly decreasing vector of non-negative integers. This we can interpret as a partition $\mu=(x_0,\ldots,x_{t-1})$ with at most $t$ parts. Its conjugate $\bar{\mu}$ is then a partition with largest part at most $t$. Moreover, $x_{t}$ is a non-negative integer divisible by $t$. While any non-empty partition $\bar{\mu}$ with largest part at most $t$ can appear, we will have to work some more to understand which integers $x_{t}$ can be paired with this partition. Let $j=m-1\mmod t$ and $\tilde{j}= t\left(\floor{\frac{m-1}{t}}+1\right)$ and write
\[
  x = \alpha_j v_m + \ldots + \alpha_{t-1} v_{\tilde{j}} + \alpha^*_0 v_{\tilde{j} + 1} + \ldots + \alpha_j^* v_{m+t}.
\]
In particular, focusing on the first $t$ rows of this system of equations, we get
\[
  \mu = \alpha_0^* b_0 + \ldots + \alpha_{j-1}^* b_{j-1} + (\alpha_j^* + \alpha_j)b_j + \alpha_{j+1}b_{j+1} + \ldots + \alpha_{t-1} b_{t-1}
\]
which means that 
\[
\mmat{\alpha_0^* \\ \vdots \\ \alpha_{j-1}^* \\\alpha_j^* + \alpha_j\\\alpha_{j+1}\\\vdots\\\alpha_{t-1}} = \mmat{\hat{\mu}_1\\\vdots\\\hat{\mu}_j\\\hat{\mu}_{j+1}\\\hat{\mu}_{j+2}\\\vdots\\\hat{\mu}_t}
\] 
where $\hat{\mu}_i$ denotes the multiplicity of the part of size $i$ in $\bar{\mu}$. 

\begin{figure}[t]
\begin{center}
\includegraphics[angle=270,width=14cm]{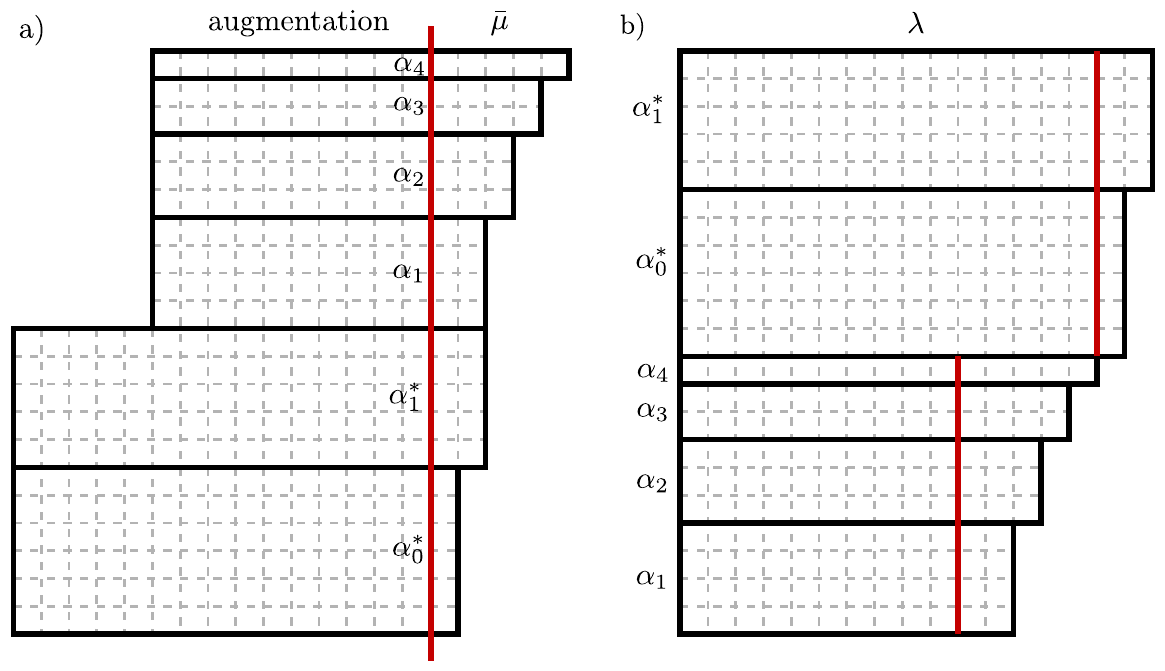}
\end{center}
 \caption[]{\label{fig:bijection} a) The pair $(\bar{\mu},x_t)$ represented as an augmented Ferrers diagram. b) The corresponding partition $f(\bar{\mu},x_t)=\lambda$ represented as a Ferrers diagram.}
\end{figure}

This is best illustrated with an example, see Figure~\ref{fig:bijection}a). Let $t=5$ and $m=12$ so that $j=1$ and $\tilde{j}=15$. Let $x=(21,16,6,3,1,53\cdot 5)$ so that $\bar{\mu}=5^1+4^2+3^3+2^9+1^6$, where we use superscripts to denote multiplicity of parts, and thus $\hat{\mu}=(6,9,3,2,1)$. Here
\[
  x = \mmat{\mu\\x_{t}} = 
  \underbrace{4}_{\alpha_1} \underbrace{\mmat{b_1\\2\cdot 5}}_{v_{12}} + 
  \underbrace{3}_{\alpha_2} \underbrace{\mmat{b_2\\2\cdot 5}}_{v_{13}} +
  \underbrace{2}_{\alpha_3} \underbrace{\mmat{b_3\\2\cdot 5}}_{v_{14}} +
  \underbrace{1}_{\alpha_4} \underbrace{\mmat{b_4\\2\cdot 5}}_{v_{15}} +
  \underbrace{6}_{\alpha_0^*} \underbrace{\mmat{b_0\\3\cdot 5}}_{v_{16}} +
  \underbrace{5}_{\alpha_1^*} \underbrace{\mmat{b_1\\3\cdot 5}}_{v_{17}}.
\]
This can be visualized, as in Figure~\ref{fig:bijection}a), by an augmented Ferrers diagram. On the right of the vertical line we have the Ferrers diagram of $\bar{\mu}$. The $\alpha_i^{(*)}$ give the multiplicity with which the part of size $i+1$ (i.e., the row of length $i+1$) appears. The part of size $(j+1)$ plays a special role in that its multiplicity is given by $\alpha_j+\alpha_j^*$. Attached to each row of $\bar{\mu}$, we have a certain multiple of $t$ which we represent by rows of additional boxes which extend to the left of the vertical line.

Returning to the question which $x_t$ are possible for a given $\mu$, we observe that for a fixed $\mu$ the only choice we have for $x_t$ arises from $\alpha_j^*+\alpha_j=\hat{\mu}_{j+1}$ given $\alpha_j^*\geq 0$ and $\alpha_j>0$. Thus, given a fixed $\mu$ the values of $x_t$ that can appear for $\mmat{\mu\\ x_t}\in C_m$ are determined by
\begin{eqnarray}
\label{eqn:xt}
  \frac{x_t}{t} - \floor{\frac{m}{t}}\left(\sum_{i=1}^t \hat{\mu}_i \right) - \left(\sum_{i=j+2}^t \hat{\mu}_i \right) &\in& \{0,\ldots,\hat{\mu}_{j+1}-1\}.
\end{eqnarray}
Thus, the set $\Lambda\cap C_m$ can be interpreted combinatorially as the set of pairs $(\bar{\mu},x_t)$ of a non-empty partition $\bar{\mu}$ with largest part at most $t$ and a number $x_t$ which satisfies (\ref{eqn:xt}).

Now that we have extracted this non-obvious definition from the geometric construction it is a straightforward matter to give a bijection $f$ between $\Lambda\cap C_m$ and $\tilde{C}_m$, as illustrated in Figure~\ref{fig:bijection}.

\noindent {\bf The Bijection:}
Given a pair $(\bar{\mu},x_t)$ in $\Lambda\cap C_m$, consider its augmented Ferrers diagram. In Figure~\ref{fig:bijection}a), the augmentations to the right of the vertical line come in two different sizes. Between these two sections we make a horizontal cut in the augmented Ferrers diagram and place the bottom part on top so that all rows are flush left, as shown in Figure~\ref{fig:bijection}b). The result is the Ferrers diagram of a partition $f(\bar{\mu},x_t)=\lambda$ where the difference between smallest and largest part is at most $t$ and the smallest part is exactly $m$. In the example, $f(\bar{\mu},x_t)=\lambda=17^5+16^6+15^1+14^2+13^3+12^4$. 

Formally, $f$ maps $(\bar{\mu},x_t)$ to the partition
\[
  \lambda = \left(j+1+t\left(\floor{\frac{m}{t}}+1\right)\right)^{\alpha_j^*} + \ldots + \left(1+t\left(\floor{\frac{m}{t}}+1\right)\right)^{\alpha_0^*} + \left(t+t\floor{\frac{m}{t}}\right)^{\alpha_{t-1}} + \ldots + \left(j+1+t\floor{\frac{m}{t}}\right)^{\alpha_{j}}.
\]

The inverse operation can be performed simply by ``cutting'' the Ferrers diagram of $\lambda$ horizontally between parts of size at least $1+t\left(\floor{\frac{m}{t}}+1\right)$ and parts of size at most $t+t\floor{\frac{m}{t}}$.  
Rearranging the (augmented) Ferrers diagram in this fashion preserves the sum of coordinates in the vectors $x$ and $\lambda$, i.e., it is \emph{height-preserving}. We summarize this result in the following lemma.

\begin{lemma}
\label{thm:bijection-piece}
The map $f$ defined above is a height-preserving bijection between $\Lambda\cap C_m$ and $\tilde{C}_m$.
\end{lemma}

We now prove Theorem ~\ref{thm:bijection}.

\begin{proof}[Proof of Theorem~\ref{thm:bijection}]
By construction, we have $\PPP_t=\bigcup_{m\geq1} \tilde{C}_m$. From Theorem~\ref{thm:main-theorem} we know $\Lambda\cap X_t=\bigcup_{m\geq 1} \Lambda\cap C_m$ where this union is disjoint. Using Lemma~\ref{thm:bijection-piece}, we obtain a height-preserving bijection between $\PPP_t$ and $\Lambda\cap X_t$ that is given piecewise between $\Lambda\cap C_m$ and $\tilde{C}_m$ for all $m\geq 1$.
\end{proof}

Note that it is not necessary to invoke Theorem~\ref{thm:main-theorem} to prove Theorem~\ref{thm:bijection}. Let $(\bar{\mu},x_t)$ be a pair consisting of a non-empty partition $\bar{\mu}$ with largest part at most $t$ and a non-negative integer $x_t$ divisible by $t$.  
To show that each such pair $(\bar{\mu},x_t)$ lies in a unique $\Lambda\cap C_m$ it suffices to observe that for every such pair there is a unique $m$ such that (\ref{eqn:xt}) holds. 
Intuitively, given $x_t$, we augment the Ferrers diagram of $\bar{\mu}$ by adding rows of boxes on the left which are subject to the following constraints:
\begin{enumerate}[label=\roman*)]
\item The number of boxes in each row has to be a multiple of $t$.
\item Only two different multiples of $t$ may appear.
\item The long rows always have to be at the bottom.
\end{enumerate} 
This has a unique solution due to the convention $\alpha_j>0$. Using this argument, it is possible to make the bijective proof entirely combinatorial.  The strength of the polyhedral geometry approach is that it provided the intuition necessary to define the $C_m$ and thus led us to this combinatorial insight.

\section{Bounded Differences vs.\ Fixed Differences}
\label{sec:bounded-vs-fixed}

A partition $\lambda$ has \emph{fixed difference $t$} if the difference between the largest part and smallest part of $\lambda$ is exactly equal to $t$. Let $\tilde{p}(n,t)$ denote the number of partitions of $n$ with fixed difference $t$ and let $\tilde{P}_t(q) = \sum_{n\geq 1}\tilde{p}(n,t)q^n$ denote the corresponding generating function. A recent paper by Andrews, Beck, and Robbins \cite{Andrews2014} proves the following result:
\begin{theorem}[Andrews--Beck--Robbins \cite{Andrews2014}]
\label{thm:andrews}
For all $t > 1$,
\begin{eqnarray}
\label{eqn:andrewsrat-fun}
\tilde{P}_t(q) = q^t\sum_{m\geq 1}\frac{q^{2m}(q)_{m-1}}{(q)_{m+t}}
=\frac{q^{t-1}(1-q)}{(1-q^{t-1})(1-q^t)} - \frac{q^{t-1}(1-q)}{(1-q^{t-1})(1-q^t)(q)_t}  + \frac{q^t}{(1-q^{t-1})(q)_t}.
\end{eqnarray}
\end{theorem}
Just as in the case of partitions with bounded differences, this formula has the surprising feature that an infinite sum of rational functions is reduced to a single rational function.  
The methods used to obtain this reduction are $q-$series arguments that include the use of $q-$binomial coefficients and an application of Heine's transformation.   
Because $\tilde P_t(q)$ is a rational function, it follows for $t>1$ that $\tilde p(n,t)$ is a quasipolynomial and closed term formulas for fixed $t$ are easily obtained similarly to (\ref{eqn:quasipolyP_2a}) and (\ref{eqn:quasipolyP_2b}) in this paper.   

Partitions with bounded differences and partition with fixed differences are related quite simply. If $t=0$, then these two notions are equivalent and, in particular, $\tilde{p}(n,0)=p(n,0)$. If $t\geq 1$, then it follows directly from the respective definitions that
\begin{eqnarray}
  \tilde{P}_t(q) = P_t(q) - P_{t-1}(q) \;\;\;\; &\text{ and }& \;\;\;\; \tilde{p}(n,t) = p(n,t) - p(n,t-1). \label{eqn:fixed}
\end{eqnarray}

\begin{figure}[t]
\begin{center}
\includegraphics[angle=270,width=14cm]{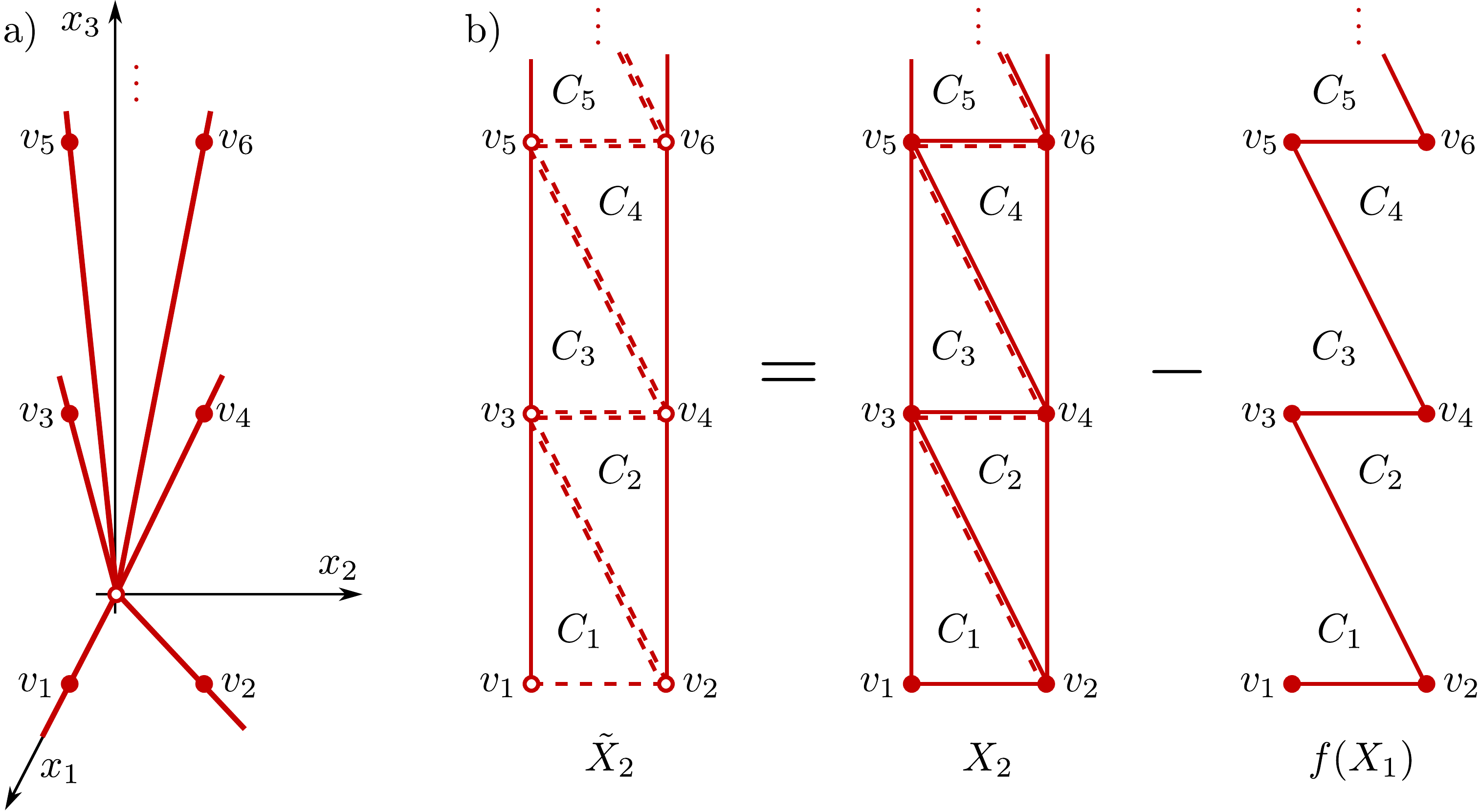}
\end{center}
 \caption[]{\label{fig:fixed} Fixed differences for $t=2$. a) The basic construction is identical to the case of bounded differences. b) In the case of fixed difference, the constituent cones have two open faces. The geometric model $\tilde{X}_2$ can thus be viewed as $X_2$ with a piecewise linear transformation of $X_1$ removed.}
\end{figure}

The inclusion-exclusion formulas (\ref{eqn:fixed}) can be visualized on the geometric level, as shown in Figure~\ref{fig:fixed}. Following the same construction as in Section~\ref{sec:polyhedral-model} we obtain a polyhedral model $\tilde{X}_t$ for $\tilde{P}_t(q)$ as an infinite union of $(t+1)$-dimensional cones that each have two open facets. In contrast to the bounded differences case, $\tilde{X}_t$ is not itself a simplicial cone with some open faces. Instead,
\[
  \tilde{X}_t = X_t \setminus f(X_{t-1})
\]
where $f$ is a piecewise linear of the form 
\[
  f(X_{t-1}) = \bigcup_{m\geq 0} f_m(C_m^{t-1})
\]
where the $C_m^{t-1}$ are the constituent simplicial cones of the model $X_{t-1}$ and the $f_m$ are unimodular linear maps, as can be seen in Figure~\ref{fig:fixed} for $t=2$. This construction shows that $\tilde{P}_t(q)$ is a rational function via a geometric argument, thus answering a question posed to the authors by George Andrews. At the same time Figure~\ref{fig:fixed} makes clear that from the geometric perspective the bounded difference setting is more natural to work with.

With the identities (\ref{eqn:fixed}) in hand, results about bounded differences can be easily converted into results about fixed differences and vice versa. In particular, only elementary arithmetic is needed to show the direct correspondence between Theorem~\ref{thm:bounded-differences} and Theorem~\ref{thm:andrews}.

\section{Conclusion}
\label{sec:conclusion}
Recalling the advances of J.J. Sylvester and others, Theorem~\ref{thm:bounded-differences}, hence, Theorem \ref{thm:andrews}, can be obtained {\it constructively, without the aid of analysis} \cite{Dickson}.
In this article we have modeled the set of partitions with difference between largest and smallest part bounded by $t$ as the set of integer points in a half-open simplicial cone in $(t+1)$-dimensional space. This is remarkable because it is not immediate from the definition of these partitions that they have a linear model in a fixed-dimensional space at all. Yet, the geometric model is surprisingly natural, given how neatly the cones $C_m$ fit together to form the simplicial cone $X_t$. In particular, this explains geometrically why $P_t(q)$ is a rational function and, more specifically, why the infinite sum of rational functions (\ref{eqn:inf-sum}) simplifies to the single rational function (\ref{eqn:rat-fun}). Moreover, the geometric construction leads naturally to a bijective proof of this identity. From a combinatorial perspective, this bijection is interesting because it is not obvious combinatorially and yet arises directly from the polyhedral model. From a geometric perspective, this bijection underlines the importance of piecewise linear transformations of polyhedral models of combinatorial counting functions. 

At least three questions for future research present themselves:
\begin{enumerate}
\item The geometric methods developed in this article can also be applied to counting partitions with specified distances as introduced in \cite{Andrews2014}. However, just as discussed in Section~\ref{sec:bounded-vs-fixed} the resulting polyhedral models will involve inclusion-exclusion, which makes a geometric treatment of specified distances a priori unwieldy. What is a good analogue of specified distances in the bounded differences setting that leads to a convex polyhedral model?

\item The inductive proof of Theorem~\ref{thm:main-theorem} can be translated into an inductive simplification of the infinite sum (\ref{eqn:inf-sum}) to the rational function (\ref{eqn:rat-fun}). What is the relation of this polyhedral construction to (anti-)telescoping methods in partition theory such as \cite{Andrews2013}? 

\item The proof of Theorem~\ref{thm:andrews} in \cite{Andrews2014} utilizes the Heine transformation. Is there a polyhedral construction that would provide a multivariate generalization of the Heine transformation? 
\end{enumerate}

Polyhedral models have proven to be a useful tool in combinatorics \cite{Breuer,Beck2006-iop,BBGM2014} and in partition theory \cite{BBKSZ1,BZ2015}. In particular, they can help in the construction of bijective proofs for partition identities, as the present article demonstrates, and even in the construction of combinatorial witnesses for partition congruences \cite{BEK-6j-1}. Polyhedral methods have great potential for further applications in this area and we look forward to more such applications in future research.

\bigskip
\paragraph{{\textbf{Acknowledgments}}} The authors wish to thank George Andrews and Peter Paule for suggesting the topic of fixed differences and asking if polyhedral methods can show that the generating function for partitions with fixed differences is a rational function.

\setlength{\bibsep}{0pt} 
\renewcommand{\bibfont}{\small} 
\bibliographystyle{abbrv}
\bibliography{references}

\end{document}